\documentclass[12pt,english,a4paper]{smfart}

\usepackage[utf8]{inputenc}
\usepackage[T1]{fontenc}
\usepackage[francais]{babel}
\usepackage{lmodern}
\usepackage{smfthm}
\usepackage[headings]{fullpage}
\usepackage{amssymb}

\let\cal\mathcal

\let\hat\widehat
\let\tilde\widetilde
\let\phi\varphi

\let\epsilon\varepsilon

\def\Q{{\bf Q}} 
\def\Z{{\bf Z}}
\def\C{{\bf C}}
\def\N{{\bf N}}

\def\A{{\bf A}}
\def\E{{\bf E}}
\def\B{{\bf B}}
\def\OO{{\cal O}}
\def\G{{\cal G}}
\def\F{{\bf F}}

\def\Ker{{\mathrm{ker}}}
\def\End{{\mathrm{End}}}

\def\D{{\bf D}}

\def\id{{\mathrm{id}}}

\def\Mat{{\mathrm{Mat}}}

\def\Lie{{\mathrm{Lie}}}
\def\pa{{\mathrm{pa}}}

\def\an{{\mathrm{an}}}
\def\LT{{\mathrm{LT}}}
\def\dR{{\mathrm{dR}}}
\def\rig{{\mathrm{rig}}}
\def\ra{{\longrightarrow}}

\newcommand{\fpa}{\text{$F\!\mbox{-}\mathrm{pa}$}}
\newcommand{\dan}{\text{$\mbox{-}\mathrm{an}$}}
\newcommand{\fan}{\text{$F\!\mbox{-}\mathrm{la}$}}
\newcommand{\fla}{\text{$F\!\mbox{-}\mathrm{la}$}}

\def\Zp{{\Z_p}}
\def\Qp{{\Q_p}}
\def\Cp{{\C_p}}

\def\Qpbar{{\overline{\Q}_p}}

\def\At{{\tilde{\bf{A}}}}
\def\Atplus{{\tilde{\bf{A}}^+}}
\def\Bt{{\tilde{\bf{B}}}}
\def\Btplus{{\tilde{\bf{B}}^+}}
\def\Et{{\tilde{\bf{E}}}}

\def\Bdr{{\bf{B}_{\mathrm{dR}}}}
\def\Bdrplus{{\bf{B}_{\mathrm{dR}}^+}}

\def\Btrig{{\Bt_{\mathrm{rig}}^{\dagger}}}
\def\Btrigplus{{\Bt_{\mathrm{rig}}^+}}

\def\Gal{{\mathrm{Gal}}}
\def\Hom{{\mathrm{Hom}}}
\def\GL{{\mathrm{GL}}}

\def\LT{{\mathrm{LT}}}

\def\la{{\mathrm{la}}}
\def\cbf{{\bf{c}}}
\def\kbf{{\bf{k}}}
\def\ubar{{\overline{u}}}

\author{Léo Poyeton}
\address{IMB\\
Université de Bordeaux}
\email{leo.poyeton@u-bordeaux.fr}
\urladdr{perso.ens-lyon.fr/leo.poyeton/}

\date{\today}

\title{$F$-analytic $B$-pairs}

\begin{document}

\begin{abstract}
In this note, we define the notion of $F$-analytic $B$-pairs and we prove that its category is equivalent to the one of $F$-analytic $(\phi_q,\Gamma_K)$-modules.
\end{abstract}

\maketitle

\tableofcontents

\setlength{\baselineskip}{18pt}

\section*{Introduction}\label{intro} 

Let $p$ be a prime and let $K$ be a finite extension of $\Qp$. One of the main tools to study $p$-adic representations of $\G_K = \Gal(\Qpbar/K)$ is to operate a ``dévissage'' of the extension $\Qpbar/K$ through an intermediate extension $K_\infty/K$ which contains most of the ramification of $\Qpbar/K$ but such that $K_\infty/K$ is nice enough (for example when $K_\infty/K$ is an infinitely ramified $p$-adic Lie extension). 

In some sense, the simplest extension one can choose for $K_\infty/K$ is the cyclotomic extension of $K$. Using the theory of fields of norms \cite{Win83} attached to the cyclotomic extension of $K$, Fontaine has constructed \cite{Fon90} a theory of cyclotomic $(\phi,\Gamma_K)$-modules, which are finite dimensional vector spaces, defined on a local field $\B_K$ which is of dimension 2, and endowed with semilinear actions of a Frobenius $\phi$ and of $\Gamma_K = \Gal(K(\mu_{p^\infty})/K)$ that commute one to another. Moreover, Fontaine has constructed a functor $V \mapsto D(V)$ which is an equivalence of categories between $p$-adic representations of $\G_K$ and étale $(\phi,\Gamma_K)$-modules (which means that $\phi$ is of slope $0$). The main theorem of \cite{cherbonnier1998representations} show that these $(\phi,\Gamma_K)$-modules are overconvergent and it allows us to relate the cyclotomic $(\phi,\Gamma_K)$-modules with classical $p$-adic Hodge theory, using the fact that the resulting overconvergent $(\phi,\Gamma_K)$-modules give rise to what we still call $(\phi,\Gamma_K)$-modules but defined on the cyclotomic Robba ring $\B_{\rig,K}^\dagger$.

The construction of the $p$-adic Langlands correspondence for $\GL_2(\Qp)$ \cite{colmez2010representations} relies heavily on this construction, and in particular on the computations made by Colmez in the trianguline case \cite{colmez2008representations}. A trianguline representation of $\G_{\Qp}$ of dimension $2$ is a $2$-dimensional $p$-adic representation of $\G_{\Qp}$ such that its $(\phi,\Gamma)$-module is an extension of two rank $1$ $(\phi,\Gamma)$-modules. Note that this does not imply that the representation itself is an extension of two rank $1$ representations because the $(\phi,\Gamma)$-modules appearing in this ``triangulation'' do not need to be étale. 

In order to extend this correspondence to $\GL_2(F)$, it seems necessary to replace the theory of cyclotomic $(\phi,\Gamma_K)$-modules by Lubin-Tate $(\phi_q,\Gamma_K)$-modules, where $F \subset K$ and $K_\infty/K$ is generated by the torsion points of a Lubin-Tate group attached to a uniformizer of $F$. Specializing Fontaine's constructions, Kisin and Ren have shown that we can attach to each representation of $\G_K$ a Lubin-Tate $(\phi_q,\Gamma_K)$-module $D(V)$ over a $2$-dimensional local field $\B_K$ (which is not the same as in the cyclotomic case) and such that $V \mapsto D(V)$ gives rise to an equivalence of categories when the image is restricted to the subcategory of étale objects.

However, unlike in the cyclotomic case, the resulting Lubin-Tate $(\phi_q,\Gamma_K)$-modules are usually not overconvergent. There still exists a subcategory of the one of Lubin-Tate $(\phi_q,\Gamma_K)$-modules in which the objets are overconvergent, which is the subcategory of $F$-analytic $(\phi_q,\Gamma_K)$-modules. The fact that those are overconvergent is the main theorem of \cite{Ber14MultiLa}. The generalization of trianguline representations in the $\Qp$-cyclotomic case to $F$-analytic representations has been studied in \cite{FX13} (and Kisin and Ren mainly studied $F$-analytic crystalline representations in \cite{KR09}). 

A generalization of trianguline representations in the cyclotomic case for $\G_K$ has been done by Nakamura in \cite{Nakapieds} using the language of Berger's $B$-pairs \cite{BerBpaires} (and their natural extension to $E$-representations which are called $E-B-$pairs in \cite{Nakapieds}) but as noted in the introduction of \cite{FX13}, this language does not appear well suited to deal with Lubin-Tate objects. 

In \cite[Rem. 10.3]{Ber14MultiLa} Berger notes that his results and methods should extend to prove that there is an equivalence of categories between $F$-analytic $(\phi_q,\Gamma_K)$-modules and $F$-analytic $B$-pairs, and it is this result this note is meant to prove.

In the cyclotomic case, it is often useful to switch between cyclotomic $(\phi,\Gamma_K)$-modules and $B$-pairs, some properties being easier to prove using one of the categories instead of the other, which is why the following theorem might prove useful:

\begin{enonce*}{Theorem A}
There is an equivalence of categories between $F$-analytic $B$-pairs and $F$-analytic $(\phi_q,\Gamma_K)$-modules.
\end{enonce*}

In particular, a recent preprint from Porat shows that for $F$-analytic $2$-dimensional representations of $\G_F$, $V$ is trianguline in the cyclotomic sense if and only if it is trianguline in the sense of \cite{FX13}. His theorem should extend for $F$-analytic representations of arbitrary dimension and it should follow as a straightforward consequence of our theorem A.  

As stated above, the usual language of $B$-pairs is not well suited to deal with Lubin-Tate objects and therefore it is quite uneasy to identify the corresponding objects through the equivalence of categories of Theorem A. Ding has constructed in \cite{Dingpartially} a variant of Berger's $B$-pairs with a Lubin-Tate flavour, which are particularly well suited to deal with the Lubin-Tate case. For any embedding $\sigma : F \to \Qpbar$, and for any $B$-pair $D$, Ding constructs what he calls a $B_\sigma$-pair $D_\sigma$, such that $D \mapsto D_\sigma$ is an equivalence of categories between $B$-pairs and $B_\sigma$-pairs. In the $F$-analytic case, we construct a functor $D \mapsto W(D)$ from the category of $F$-analytic $(\phi_q,\Gamma_K)$-modules to the category of $F$-analytic $B_{\id}$-pairs and which is the natural Lubin-Tate analogue of the constructions of Berger \cite{BerBpaires}. In particular, the following ensues from theorem A but the correspondence between objects is easier to see:

\begin{enonce*}{Theorem B}
The functor $D \mapsto W(D)$, from the category of $F$-analytic $(\phi_q,\Gamma_K)$-modules to the category of $F$-analytic $B_{\id}$-pairs is an equivalence of categories.
\end{enonce*}

\section*{Structure of the note}
The first three sections of this note are meant to recall the setting, notations and few properties of Lubin-Tate extensions, $(\phi_q,\Gamma_K)$-modules and locally analytic vectors from \cite{Ber14MultiLa} that are needed for the rest of this note. In particular, these are pretty much the same as \cite[\S 1, 2 and 3]{Ber14MultiLa}. Section 4 explains the notion of $F$-analyticity in the case of $F$-representations and $(\phi_q,\Gamma_K)$-modules. In section 5, we recall the notion of $(B,E)$-pairs, define $F$-analyticity for $(B,E)$-pairs and prove the main theorem of this note. In section 6 we explain how to replace the category of $F$-analytic $B$-pairs by more natural objects, following constructions from Yiwen Ding. 

\section*{Acknowledgements}
The first version of this note (mainly the first five sections of the current note) was written between December 2019 and January 2020. The author would like to thank Laurent Berger and Yiwen Ding for some useful discussions resulting in the first version of this note. The author would also like to thank Gal Porat for discussing the results from this note and his remarks and questions which resulted in this new version. 

\section{Lubin-Tate extensions}
\label{section LT}
Let $F$ be a finite extension of $\Qp$, let $\OO_F, \pi$ and $k_F$ denote respectively its ring of integers, a uniformizer of $\OO_F$ and its residue field. Let $h \geq 1$ be such that $|k_F|=q=p^h$. We let $F_0= W(k_F)[1/p]$, the maximal unramified extension of $\Qp$ inside $F$, and we let $e$ to be the ramification index of $F$. Let $\Sigma$ be the set of embeddings of $F$ in $\Qpbar$, and let $\sigma$ be the absolute Frobenius on $F_0$. For $\tau \in \Sigma$, there exists a unique $n(\tau) \in \{0,\ldots,h-1\}$ such that $\overline{\tau} = \overline{\sigma}^{n(\tau)}$ on $k_F$. We also let $E$ to be a field of coefficients which is a finite Galois extension of $\Qp$ containing $F$ (hence $F^\Gal$), and write $\Sigma_E$ for $\Gal(E/\Qp)$. We also let $\Sigma_0 = \Sigma \setminus \{\id\}$.

Let $S$ be a formal $\OO_F$-module Lubin-Tate group law attached to $\pi$, such that the endomorphism of multiplication by $\pi$ is given by the power series $[\pi](T) = T^q+\pi T$. For $a \in \OO_F$, we will denote $[a](T)$ the power series giving the endomorphism of multiplication by $a$ for $S$. Let $F_n$ be the field generated by $F$ and the points of $\pi^n$-torsion, that is the roots of $[\pi^n](T)$. Let $F_\infty = \bigcup_{n\geq 1}F_n$, $\Gamma_F = \Gal(F_\infty/F)$ and $H_F = \Gal(F_\infty/F)$. Let $\chi_\pi$ be the attached Lubin-Tate character, so that the map $g \mapsto \chi_\pi(g)$ induces an isomorphism between $\Gamma_F$ and $\OO_F^\times$. Note that there exists an unramified character $\eta : \G_F \to \Z_p^\times$ such that $N_{F/\Qp}(\chi_\pi)=\eta\chi_{\mathrm{cycl}}$, where $\chi_{\mathrm{cycl}}$ is the cyclotomic character. 

If $K$ is a finite extension of $F$, we write $K_n = KF_n$ and $K_\infty = KF_\infty$. We let $\Gamma_K = \Gal(K_\infty/K)$ and $H_K = \Gal(K_\infty/K)$. We let $K_\infty^\eta=\Qpbar^{\Ker \eta\chi_{\mathrm{cycl}}}$, so that $K_\infty^\eta \subset K_\infty$ and that $\eta\chi_{\mathrm{cycl}}$ identifies $\Gal(K_\infty^\eta/K)$ with an open subgroup of $\Z_p^\times$.

Now let $\Gamma_n = \Gal(K_\infty/K_n)$ so that $\Gamma_n = \{ g \in \Gamma_K$ such that $\chi_\pi(g) \in 1+\pi^n \OO_F\}$. Let $u_0 = 0$ and for each $n \geq 1$, chose $u_n \in \Qpbar$ such that $[\pi](u_n)=u_{n-1}$, with $u_1 \neq 0$. We have $v_p(u_n) = 1/q^{n-1}(q-1)e$ for $n \geq 1$ and $F_n=F(u_n)$. We also let $Q_n(T)$ be the minimal polynomial of $u_n$ over $F$, so that $Q_0(T)=T$, $Q_1(T)=[\pi](T)/T$ and $Q_{n+1}(T)=Q_n([\pi](T))$ if $n \geq 1$. Let $\log_{\LT}(T) = T + O(\deg \geq 2) \in F[\![T]\!]$ denote the Lubin-Tate logarithm map, which converges on the  open unit disk and satisfies $\log_{\LT}([a](T)) = a  \cdot \log_{\LT}(T)$ if $a\in \OO_F$.  Note that we have $\log_{\LT}(T) = T \cdot \prod_{k \geq 1} Q_k(T)/\pi$. We also let $\exp_{\LT}(T)$ denote the inverse of $\log_{\LT}(T)$.

Let $\OO_{\Cp}^{\flat} = \{ (x_0,x_1,\hdots)$, with $x_n \in \OO_{\Cp}/\pi$ and such that $x_{n+1}^q=x_n$ for all $n \geq 0\}$. This ring is endowed with the valuation $v_{\E}(\cdot)$ defined by $v_{\E}(x) = \lim_{n \to +\infty} q^n v_p(\hat{x}_n)$ where $\hat{x}_n \in \OO_{\Cp}$ lifts $x_n$. The ring $\OO_{\Cp}^{\flat}$ is complete for $v_{\E}(\cdot)$. If the $\{u_n\}_{n \geq 0}$ are as above, then $\ubar = (\ubar_0,\ubar_1,\hdots) \in \OO_{\Cp}^{\flat}$ and $v_{\E}(\ubar) = q/(q-1)e$. Let $\Cp^{\flat}$ be the fraction field of $\OO_{\Cp}^{\flat}$.

Let $W_F(\cdot)=\OO_F \otimes_{\OO_{F_0}} W(\cdot)$ be the $F$-Witt vectors. Let $\Atplus=\OO_F \otimes_{\OO_{F_0}}W(\OO_{\Cp}^{\flat})$, $\At=\OO_F \otimes_{\OO_{F_0}}W(\Cp^{\flat})$ and let $\Btplus=\Atplus[1/\pi]$ and $\Bt=\At[1/\pi]$. These rings are preserved by the Frobenius map $\phi_q=\id \otimes \phi^h$. By \cite[\S 9.2]{Col02}, there exists $u \in \Atplus$, whose image in $\OO_{\Cp}^{\flat}$ is $\ubar$, and such that $\phi_q(u) = [\pi](u)$ and $g(u) = [\chi_\pi(g)](u)$ if $g \in \Gamma_F$.

Every element of $\Btplus[1/[\ubar]]$ can be written uniquely as a sum $\sum_{k \gg -\infty} \pi^k [x_k]$ where $\{x_k\}_{k \in \Z}$ is a bounded sequence of $\Et$. For $r \geq 0$, we define a valuation $V(\cdot,r)$ on $\Btplus[1/[\ubar]]$ by 
\[ V(x,r) = \inf_{k \in \Z} \left(\frac{k}{e} + \frac{p-1}{pr} v_{\E}(x_k)\right)  \text{ if } x = \sum_{k \gg -\infty} \pi^k [x_k]. \]

If $I$ is a closed subinterval of $[0;+\infty[$, then let $V(x,I) = \inf_{r \in I} V(x,r)$. We define $\Bt^I$ to be the completion of $\Btplus[1/[\ubar]]$ for the valuation $V(\cdot,I)$ if $0 \notin I$. If $I=[0;r]$, then let $\Bt^I$ be the completion of $\Btplus$ for $V(\cdot,I)$.

For $\rho>0$, let $\rho'=\rho \cdot e \cdot p/(p-1) \cdot (q-1)/q$ as in \cite[\S 3]{Ber14MultiLa}. We have $V(u^i,r) = i/r'$ for $i \in \Z$ if $r>1$ (see \cite[\S 3]{Ber14MultiLa}).

Let $I$ be either a subinterval of $]1;+\infty[$ or such that $0 \in I$, and let $f(Y) = \sum_{k \in \Z} a_k Y^k$ be a power series with $a_k \in F$ and such that $v_p(a_k)+k/\rho' \to +\infty$ when $|k| \to +\infty$ for all $\rho \in I$. The series $f(u)$ converges in $\Bt^I$ and we let $\B_F^I$ denote the set of $f(u)$ where $f(Y)$ is as above. It is a subring of $\Bt_F^I = (\Bt^I)^{H_F}$, which is stable under the action of $\Gamma_F$. The Frobenius map gives rise to a map $\phi_q : \B_F^I \to \B_F^{qI}$. If $m \geq 0$, then we have $\phi_q^{-m}(\B_F^{q^m I}) \subset \Bt_F^I$ and we let $\B_{F,m}^I = \phi_q^{-m}(\B_F^{q^m I})$. 

We will write $\B_{\mathrm{rig},F}^{\dagger,r}$ for $\B_F^{[r;+\infty[}$. Let $\B_F^{\dagger,r}$ denote the set of $f(u) \in \B_{\mathrm{rig},F}^{\dagger,r}$ such that the sequence $\{a_k\}_{k \in \Z}$ is bounded. Let $\B_F^{\dagger} = \cup_{r \gg 0} \B_F^{\dagger,r}$. Its residue field $\E_F$ is isomorphic to $\F_q(\!(\overline{u})\!)$. If $K$ is a finite extension of $F$ then by the theory of the field of norms (see \cite{Win83}), there corresponds to $K/F$ a separable extension $\E_K / \E_F$, of degree $[K_\infty:F_\infty]$. Since $\B_F^{\dagger}$ is a Henselian field, there exists a finite unramified extension $\B_K^{\dagger}/ \B_F^{\dagger}$ of degree $f = [K_\infty:F_\infty]$ whose residue field is $\E_K$ (see \S 2 and \S 3 of \cite{matsuda1995local}). There exist therefore $r(K) >0$ and elements $x_1,\hdots,x_f$ in $\B_K^{\dagger,r(K)}$ such that $\B_K^{\dagger,s} = \oplus_{i=1}^f \B_F^{\dagger,s} \cdot x_i$ for all $s \geq r(K)$. Note that the rings $\B_K^{\dagger}$ are actually contained inside $\Bt$. We also let $\B_K$ to be the $p$-adic completion of $\B_K^{\dagger}$ inside $\Bt$, and $\A_K$ its ring of integers for the $p$-adic topology (note that we could have defined $\A_F$ as the $p$-adic completion of $\OO_F[\![u]\!][1/u]$ inside $\At$, put $\B_F = \A_F[1/\pi]$ and used the same argument as in the beginning of \cite[\S 6.1]{colmez2008espaces} to define $\B_K$). Let $\B$ be the $p$-adic completion of $\bigcup_{K/F}\B_K$ inside $\Bt$.

Let $\B_{\mathrm{rig},K}^{\dagger,r}$ denote the Fréchet completion of $\B_{K}^{\dagger,r}$ for the valuations $\{V(\cdot,[r;s]) \}_{s \geq r}$. Let $\B_{\mathrm{rig},K,m}^{\dagger,r} = \phi_q^{-m}(\B_{\mathrm{rig},K}^{\dagger,q^mr})$ and $\B_{\mathrm{rig},K,\infty}^{\dagger,r} = \cup_{m \geq 0} \B_{\mathrm{rig},K,m}^{\dagger,r}$. Let $\Bt_{\mathrm{rig}}^{\dagger,r}$ denote the Fréchet completion of $\Btplus[1/[\ubar]]$ for the valuations $\{V(\cdot,[r;s]) \}_{s \geq r}$. Let $\Bt_{\rig}^\dagger = \cup_{r \gg 0} \Bt_{\mathrm{rig}}^{\dagger,r}$, $\Bt_{\mathrm{rig},K}^{\dagger,r} = (\Bt_{\mathrm{rig}}^{\dagger,r})^{H_K}$ and $\Bt_{\mathrm{rig},K}^\dagger = (\Btrig)^{H_K}$.

Recall that $K_\infty^\eta/K$ is the extension of $K$ attached to $\eta\chi_{\mathrm{cycl}}$. Let $\Gamma_K' = \Gal(K_\infty^\eta/K)$. Let $\B_{K,\eta}^\dagger$, $\B_{K,\eta}^I$ and $\B_{\mathrm{rig},K,\eta}^\dagger$ be as in \cite[\S 8]{Ber14MultiLa}. By the same arguments as in \cite[\S 8]{Ber14MultiLa}, there is an equivalence of categories between étale $(\phi,\Gamma_K')$-modules over $E \otimes_{\Qp}\B_{\mathrm{rig},K,\eta}^\dagger$ (it is also true over $E \otimes_{\Qp}\B_{K,\eta}^\dagger$) and $E$-representations of $\G_K$. We will also denote by $\Bt_{\mathrm{rig},\eta}^\dagger$ the ring $\Bt_{\mathrm{rig}}^\dagger$ in the specific case of $F=\Qp$, so that $\Bt_{\mathrm{rig}}^\dagger = F \otimes_{F_0} \Bt_{\mathrm{rig},\eta}^\dagger$. Note that the ring $\Bt_{\mathrm{rig},\eta}^\dagger$ does actually not depend on $\eta$ but we use this notation for convenience.

A $(\phi_q,\Gamma_K)$-module over $\B_K$ is a $\B_K$-vector space $\D$ of finite dimension $d$, along with a semilinear Frobenius map $\phi_q$ and a commuting continuous and semilinear action of $\Gamma_K$. We say that $\D$ is étale if there exists a basis of $\D$ in which $\Mat(\phi)$ belongs to $\GL_d(\A_K)$. By specializing the constructions of \cite{Fon90}, Kisin and Ren prove the following theorem \cite[Thm. 1.6]{KR09}.

\begin{theo}\label{fontkisren}
The functors $V \mapsto (\B \otimes_F V)^{H_K}$ and $\D \mapsto (\B \otimes_{\B_K} \D)^{\phi_q=1}$ give rise to mutually inverse equivalences of categories between the category of $F$-linear representations of $\G_K$ and the category of étale $(\phi_q,\Gamma_K)$-modules over $\B_K$.
\end{theo}

We say that a $(\phi_q,\Gamma_K)$-module $\D$ is overconvergent if there exists a basis of $\D$ in which the matrices of $\phi_q$ and of all $g \in \Gamma_K$ have entries in $\B_K^{\dagger}$. This basis generates a $\B_K^{\dagger}$-vector space $\D^{\dagger}$ which is canonically attached to $\D$. Theorem \ref{fontkisren} extends more generally to an equivalence of categories between the category of $E$-linear representations of $\G_K$ and the category of étale $(\phi_q,\Gamma_K)$-modules over $E \otimes_F \B_K$.

The main result of \cite{cherbonnier1998representations} states that if $F=\Qp$,  then every étale $(\phi_q,\Gamma_K)$-module over $\B_K$ is overconvergent but if $F \neq \Qp$, then this is no longer the case (see \cite{FX13} for example).


\begin{theo}
\label{struclocana}
Let $I=[r_\ell;r_k]$ with $\ell \leq k$, let $K$ be a finite extension of $F$, and let $m \geq 0$ be such that $t_\pi$ and $t_\pi/Q_k$ belong to $(\Bt^I_F)^{\Gamma_{m+k}\dan,\fla}$.
\begin{enumerate}
\item $(\Bt^I_F)^{\Gamma_{m+k}\dan,\fla} \subset \B^I_{F,m}$;
\item $(\Bt^I_K)^{\fla} = \B^I_{K,\infty}$;
\item $(\Bt_{\mathrm{rig},K}^{\dagger,r_\ell})^{\fpa} = \B_{\mathrm{rig},K,\infty}^{\dagger,r_\ell}$.
\end{enumerate}
\end{theo}
\begin{proof}
This is \cite[Thm. 4.4]{Ber14MultiLa}.
\end{proof}

\section{Locally, pro-analytic and $F$-analytic vectors}
In this section, we recall the theory of locally analytic vectors of Schneider and Teitelbaum \cite{schneider2002bis} but here we follow the constructions of Emerton \cite{emerton2004locally} as in \cite{Ber14MultiLa}. We also define the notion of $F$-analytic vectors relative to the Galois group of a Lubin-Tate extension, following the definitions of \cite{Ber14MultiLa}. We will use the following multi-index notations: if $\cbf = (c_1, \hdots,c_d)$ and $\kbf = (k_1,\hdots,k_d) \in \N^d$ (here $\N=\Z^{\geq 0}$), then we let $\cbf^\kbf = c_1^{k_1} \cdot \ldots \cdot c_d^{k_d}$.

Let $G$ be a $p$-adic Lie group, and let $W$ be a $\Qp$-Banach representation of $G$. Let $H$ be an open subgroup of $G$ such that there exists coordinates $c_1,\cdots,c_d : H \to \Zp$ giving rise to an analytic bijection $\cbf : H \to \Z_p^d$. We say that $w \in W$ is an $H$-analytic vector if there exists a sequence $\left\{w_{\kbf}\right\}_{\kbf \in \N^d}$ such that $w_{\kbf} \rightarrow 0$ in $W$ and such that $g(w) = \sum_{\kbf \in \N^d}\cbf(g)^{\kbf}w_{\kbf}$ for all $g \in H$. We let $W^{H\dan}$ be the space of $H$-analytic vectors. This space injects into $\cal{C}^{\an}(H,W)$, the space of all analytic functions $f : H \to W$.  Note that $\cal{C}^{\an}(H,W)$ is a Banach space equipped with its usual Banach norm, so that we can endow $W^{H\dan}$ with the induced norm, that we will denote by $||\cdot ||_H$. With this definition, we have $||w||_H = \sup_{\kbf \in \N^d}||w_{\kbf}||$ and $(W^{H\dan},||\cdot||_H)$ is a Banach space.

We say that a vector $w$ of $W$ is locally analytic if there exists an open subgroup $H$ as above such that $w \in W^{H\dan}$. Let $W^{\la}$ be the space of such vectors, so that $W^{\la} = \bigcup_{H}W^{H\dan}$, where $H$ runs through a sequence of open subgroups of $G$. The space $W^{\la}$ is naturally endowed with the inductive limit topology, so that it is an LB space. Note that in the Lubin-Tate setting, we have $W^{\la} = \bigcup_{n \in \N}W^{\Gamma_n-\an}$.

%
%

Let $W$ be a Fréchet space whose topology is defined by a sequence $\left\{p_i\right\}_{i \geq 1}$ of seminorms. Let $W_i$ be the Hausdorff completion of $W$ at $p_i$, so that $W = \varprojlim\limits_{i \geq 1}W_i$. The space $W^{\la}$ can be defined but as stated in \cite{Ber14MultiLa}, this space is too small in general for what we are interested in, and so we make the following definition, following \cite[Def. 2.3]{Ber14MultiLa}:

\begin{defi}
If $W = \varprojlim\limits_{i \geq 1}W_i$ is a Fréchet representation of $G$, then we say that a vector $w \in W$ is pro-analytic if its image $\pi_i(w)$ in $W_i$ is locally analytic for all $i$. We let $W^{\pa}$ denote the set of all pro-analytic vectors of $W$. 
\end{defi}

We extend the definition of $W^{\la}$ and $W^{\pa}$ for LB and LF spaces respectively. 

\begin{prop}
\label{lainla and painpa}
Let $G$ be a $p$-adic Lie group, let $B$ be a Banach $G$-ring and let $W$ be a free $B$-module of finite rank, equipped with a compatible $G$-action. If the $B$-module $W$ has a basis $w_1,\ldots,w_d$ in which $g \mapsto \Mat(g)$ is a globally analytic function $G \to \GL_d(B) \subset M_d(B)$, then
\begin{enumerate}
\item $W^{H\dan} = \bigoplus_{j=1}^dB^{H\dan}\cdot w_j$ if $H$ is a subgroup of $G$;
\item $W^{\la} = \bigoplus_{j=1}^dB^{\la}\cdot w_j$.
\end{enumerate}
Let $G$ be a $p$-adic Lie group, let $B$ be a Fréchet $G$-ring and let $W$ be a free $B$-module of finite rank, equipped with a compatible $G$-action. If the $B$-module $W$ has a basis $w_1,\ldots,w_d$ in which $g \mapsto \Mat(g)$ is a pro-analytic function $G \to \GL_d(B) \subset M_d(B)$, then
$$W^{\pa} = \bigoplus_{j=1}^dB^{\pa}\cdot w_j.$$
\end{prop}
\begin{proof}
The part for Banach rings is proven in \cite[Prop. 2.3]{Ber14SenLa} and the one for Fréchet rings is proven in \cite[Prop. 2.4]{Ber14MultiLa}.
\end{proof}

The map $\ell : g \mapsto \log_p \chi_\pi(g)$ gives an $F$-analytic isomorphism between $\Gamma_n$ and $\pi^n \OO_F$ for $n \gg 0$. If $W$ is an $F$-linear Banach representation of $\Gamma_K$ and $n \gg 0$, then we say, following \cite{Ber14MultiLa}, that an element $w \in W$ is $F$-analytic on $\Gamma_n$ if there exists a sequence $\{w_k\}_{k \geq 0}$ of elements of $W$ with $\pi^{nk} w_k \to 0$ such that $g(w) = \sum_{k \geq 0} \ell(g)^k w_k$ for all $g \in \Gamma_n$. Let $W^{\Gamma_n\dan,\fla}$ denote the space of such elements. Let $W^{\fla} = \bigcup_{n \geq 1} W^{\Gamma_n\dan,\fla}$. 

\begin{lemm}
\label{danfla}
We have $W^{\Gamma_n\dan,\fan} = W^{\Gamma_n\dan} \cap W^{\fan}$. 
\end{lemm}
\begin{proof}
See \cite[Lemm. 2.5]{Ber14MultiLa}.
\end{proof}

If $\tau \in \Sigma$, we let $\nabla_\tau$ denote the derivative in the direction $\tau$, which belongs to $E \otimes_{\Qp} \mathrm{Lie}(\Gamma_F)$. It can be defined as follows: the $E$-vector space $\Hom_{\Qp}(F,E)$ is generated by the elements of $\Sigma$. If $W$ is an $E$-linear Banach representation of $\Gamma_K$ and if $w \in W^{\la}$ and $g \in \Gamma_K$, then there exists elements $\{ \nabla_\tau \}_{\tau \in \Sigma}$ of $F^{\Gal} \otimes_{\Qp} \Lie(\Gamma_F)$ such that we can write 
$$ \log g (w) = \sum_{\tau \in \Sigma} \tau(\ell(g)) \cdot \nabla_\tau(w). $$

With the same notation, there exist $m \gg 0$ and elements $\{w_\kbf \}_{\kbf \in \N^{\Sigma}}$ such that if $g \in \Gamma_m$, then $g (w) = \sum_{\kbf \in \N^{\Sigma}} \ell(g)^\kbf w_\kbf$, where $\ell(g)^\kbf = \prod_{\tau \in \Sigma} \tau \circ \ell(g)^{k_\tau}$. We have $\nabla_\tau(w) = w_{\mathbf{1}_\tau}$ where $\mathbf{1}_\tau$ is the $\Sigma$-tuple whose entries are $0$ except the $\tau$-th one which is $1$. If $\kbf \in\N^{\Sigma}$, and if we set  $\nabla^\kbf(w) =  \prod_{\tau \in \Sigma} \nabla_{\tau}^{k_\tau} (w)$, then $w_\kbf = \nabla^\kbf(w)/\kbf!$. 

\begin{rema}
\label{fla=nablatau nul}
If $w \in W^{\la}$, then $w \in W^{\fla}$ if and only if $\nabla_\tau(w) = 0$ for all $\tau \in \Sigma \setminus \{ \id \}$.
\end{rema}

%

\section{$F$-analyticity}
We now explain what it means for a representation or a $(\phi_q,\Gamma_K)$-module to be $F$-analytic, and we give a few related properties.

We say, following \cite[\S 7]{Ber14MultiLa} that an $F$-linear representation $V$ of $\G_K$ is $F$-analytic if $\Cp \otimes_F^\tau V$ is the trivial $\Cp$-semilinear representation of $\G_K$ for all embeddings $\tau \neq \id \in \Sigma$.

The following lemma shows that the condition for an $E$-representation to be $F$-analytic depends only on the restriction of the elements of $\Sigma_E$ to $F$.

\begin{lemm}
\label{fanelin}
If $V$ is an $E$-representation of $\G_K$, then the following are equivalent:
\begin{enumerate}
\item $V$ seen as an $F$-representation is $F$-analytic;
\item $\Cp \otimes_E^g V$ is the trivial $\Cp$-semilinear representation of $\G_K$ for all $g \in \Gal(E/\Qp)$ such that $g{|_F} \neq \id$.
\end{enumerate}
\end{lemm}
\begin{proof}
See \cite[Lemm. 7.2]{Ber14MultiLa}.
\end{proof}

If $\D_{\mathrm{rig}}^{\dagger}$ is a $(\phi_q,\Gamma_K)$-module over $\B_{\mathrm{rig},K}^{\dagger}$, and if $g \in \Gamma_K$ is close enough to $1$, then the series $\log(g) = \log(1+(g-1))$ gives rise to a differential operator $\nabla_g : \D_{\mathrm{rig}}^{\dagger} \to \D_{\mathrm{rig}}^{\dagger}$. The map $\mathrm{Lie}\,\Gamma_K \to \End(\D_{\mathrm{rig}}^{\dagger})$ arising from $v \mapsto \nabla_{\exp(v)}$ is $\Qp$-linear, and we say, following \cite[\S 2.1]{KR09}, \cite[\S 1.3]{FX13} and \cite[\S 7]{Ber14MultiLa}, that $\D_{\mathrm{rig}}^{\dagger}$ is $F$-analytic if this map is $F$-linear. This is the same as asking the elements of $\D_{\mathrm{rig}}^{\dagger}$ to be pro-$F$-analytic vectors for the action of $\Gamma_K$.

Berger proved \cite[Thm. 10.1]{Ber14MultiLa} that the Lubin-Tate $(\phi_q,\Gamma_K)$-modules attached to $F$-analytic representations are overconvergent. He showed how to attach to an $E$-representation $V$ of $\G_K$ that is $F$-analytic an étale $(\phi_q,\Gamma_K)$-module $\D_{\mathrm{rig}}^\dagger(V)$ over $\B_{\mathrm{rig},K}^\dagger$ which is $F$-analytic and proved \cite[Thm. 10.4]{Ber14MultiLa} that the functor $V \mapsto \D_{\mathrm{rig}}^\dagger(V)$ gives rise to an equivalence of categories between the category of $F$-analytic $E$-representations of $G_K$ and the category of étale $F$-analytic Lubin-Tate $(\phi_q,\Gamma_K)$-modules over $E \otimes_F \B_{\mathrm{rig},K}^{\dagger}$.

We will give an equivalent condition for a Lubin-Tate $(\phi_q,\Gamma_K)$-module over $E \otimes_F \B_{\mathrm{rig},K}^{\dagger}$ to be $F$-analytic but first we recall a few of Berger's results about locally analytic periods inside some of the rings we have defined.

Given $\tau \in \Sigma$ and $f(Y) =\sum_{k \in \Z} a_k Y^k$ with $a_k \in F$, let $f^\tau(Y)=\sum_{k \in \Z} \tau(a_k) Y^k$. For $\tau \in \Sigma$, let $\tilde{n}(\tau)$ be the lift of $n(\tau) \in \Z/h\Z$ belonging to $\{0,\hdots,h-1\}$. Recall that $E$ is a finite extension of $F$ that contains $F^{\Gal}$ and that if $\tau \in \Sigma$, then we have $\nabla_\tau \in E \otimes_F \Lie(\Gamma_F)$. The field $E$ is a field of coefficients, so that $\G_K$ acts $E$-linearly below.

Let $t_\pi = \log_{\LT}(u) \in \B_{\mathrm{rig},K}^+$. Note that we actually have $t_\pi \in \B_{\mathrm{rig},F}^+$, and that $\phi_q(t_\pi) = \pi t_\pi$ and $g(t_\pi) = \chi_\pi(g) t_\pi$ if $g \in \G_F$. Let $y_\tau = (\tau \otimes \phi^{\tilde{n}(\tau)}) (u) \in \OO_E \otimes_{\OO_F} \Atplus$. We have $g(y_\tau) = [\chi_\pi(g)]^\tau(y_\tau)$ and $\phi_q(y_\tau) = [\pi]^\tau (y_\tau) = \tau (\pi) y_\tau+y_\tau^q$. Let $t_\tau = (\tau \otimes \phi^{\tilde{n}(\tau)})(t_\pi) = \log_{\LT}^\tau(y_\tau)$. 

We have $\nabla_\tau (y_\tau) = t_\tau \cdot v_\tau$ where $v_\tau = (\partial (T \oplus_{\LT} U)/ \partial U)^\tau(y_\tau,0)$ is a unit (see \S 2.1 of \cite{KR09}). Let $\partial_\tau = t_\tau^{-1} v_\tau^{-1} \nabla_\tau$ so that $\partial_\tau (y_\tau) = 1$. If $\tau,\upsilon \in \Sigma$, then $\partial_\tau \circ \partial_\upsilon = \partial_\upsilon \circ \partial_\tau$, and $\partial_\tau(y_\upsilon) = 0$ if $\tau \neq \upsilon$.

\begin{lemm}\label{partial stable}
We have $\partial_\tau((E \otimes_F \Bt_{\mathrm{rig},K}^\dagger)^{\pa}) \subset (E \otimes_F \Bt_{\mathrm{rig},K}^\dagger)^{\pa}$.
\end{lemm}
\begin{proof}
See \cite[Lemm. 5.2]{Ber14MultiLa}.
\end{proof}

\begin{prop}
\label{prop eq phigamma fan}
Let $M$ be a $(\phi_q,\Gamma_K)$-module over $E\otimes_F(\Bt_{\mathrm{rig},K}^{\dagger})^{\pa}$. Let $$\mathrm{Sol}(M) = \left\{x \in M \textrm{ such that } \nabla_\tau(x) = 0 \textrm{ for all } \tau  \in \Sigma_0 \right\}.$$

If for all $\tau \in \Sigma_0$, $\nabla_\tau(M) \subset t_\tau \cdot M$, then there exists a unique $(\phi_q,\Gamma_K)$-module $\D_{\mathrm{rig}}^\dagger$ over $E\otimes_F\B_{\mathrm{rig},K}^{\dagger}$ such that $\mathrm{Sol}(M) = (E \otimes_F (\Bt_{\mathrm{rig},K}^\dagger)^{\fpa}) \otimes_{E \otimes_F \Bt_{\mathrm{rig},K}^\dagger} \D_{\mathrm{rig}}^\dagger$ and such that $M = (E \otimes_F (\Bt_{\mathrm{rig},K}^\dagger)^{\pa})\otimes_{E\otimes_F\B_{\mathrm{rig},K}^{\dagger}} \D_{\mathrm{rig}}^\dagger$, and $\D_{\mathrm{rig}}^\dagger$ is an $F$-analytic $(\phi_q,\Gamma_K)$-module.

Moreover, if $\D$ is a $(\phi_q,\Gamma_K)$-module over $E\otimes_F\B_{\mathrm{rig},K}^{\dagger}$, and if $M = (E \otimes_F \Bt_{\mathrm{rig},K}^\dagger)\otimes_{E\otimes_F\B_{\mathrm{rig},K}^{\dagger}} \D$, then $\D$ is $F$-analytic if and only if for all $\tau \in \Sigma_0$, $\nabla_\tau(M^{\pa}) \subset t_\tau \cdot M^{\pa}$. 
\end{prop}
\begin{proof}
We first prove the first part of the theorem. Let $M$ be a $(\phi_q,\Gamma_K)$-module over $E\otimes_F(\Bt_{\mathrm{rig},K}^{\dagger})^{\pa}$. Theorem 6.1 of \cite{Ber14MultiLa} shows that 
$$\mathrm{Sol}(M) = \left\{x \in M \textrm{ such that } \nabla_\tau(x) = 0 \textrm{ for all } \tau  \in \Sigma_0 \right\}$$
is a free $E \otimes_F (\Bt_{\mathrm{rig},K}^\dagger)^{\fpa}$-module of rank $d$ such that
$$(E \otimes_F \Bt_{\mathrm{rig},K}^\dagger) \otimes_{E \otimes_F (\Bt_{\mathrm{rig},K}^\dagger)^{\fpa})^{\fpa}} \mathrm{Sol}(M) = (E \otimes_F \Bt_{\mathrm{rig},K}^\dagger) \otimes_E \D.$$
By (3) of theorem \ref{struclocana}, we have $(\Bt_{\mathrm{rig},K}^\dagger)^{\fpa} = \B_{\mathrm{rig},K,\infty}^\dagger = \bigcup_{n \geq 0} \B_{\mathrm{rig},K,n}^\dagger$. Since $\Gamma_K$ is topologically of finite type, there exist $n \geq 0$, and a basis $s_1,\hdots,s_d$ of $\mathrm{Sol}(M)$ such that $\Mat(\phi_q) \in \GL_d(E \otimes_F \B_{\mathrm{rig},K,n}^\dagger)$ and $\Mat(g) \in \GL_d(E \otimes_F \B_{\mathrm{rig},K,n}^\dagger)$ for all $g \in \Gamma_K$. If $\D_{\mathrm{rig}}^\dagger = \oplus_{i=1}^d (E \otimes_F \B_{\mathrm{rig},K}^\dagger) \cdot \phi_q^n(s_i)$, then $\D_{\mathrm{rig}}^\dagger$ is a $(\phi_q,\Gamma_K)$-module over $E \otimes_F \B_{\mathrm{rig},K}^\dagger$ such that $\mathrm{Sol}(M) = (E \otimes_F (\Bt_{\mathrm{rig},K}^\dagger)^{\fpa}) \otimes_{E \otimes_F \B_{\mathrm{rig},K}^\dagger} \D_{\mathrm{rig}}^\dagger$. 

The module $\D_{\mathrm{rig}}^\dagger$ is uniquely determined by this condition: if there are two such modules and if $X$ denotes the change of basis matrix and $P_1$, $P_2$ denote the matrices of $\phi_q$, then $X \in \GL_d(E \otimes_F \B_{\mathrm{rig},K,n}^\dagger)$ for $n \gg 0$, and the equation $X=P_2^{-1}\phi(X)P_1$ implies that $X \in \GL_d(E \otimes_F \B_{\mathrm{rig},K}^\dagger)$. 

Since $\mathrm{Sol}(M)$ is a free $E \otimes_F (\Bt_{\mathrm{rig},K}^\dagger)^{\fpa}$-module, $\D_{\mathrm{rig}}^\dagger$ is also free of the same rank.

Now, let $\D$ be a $(\phi_q,\Gamma_K)$-module over $E\otimes_F\B_{\mathrm{rig},K}^{\dagger}$, such that $M = (E \otimes_F \Bt_{\mathrm{rig},K}^\dagger)^{\pa}\otimes_{E\otimes_F\B_{\mathrm{rig},K}^{\dagger}} \D$ is such that for all $\tau \in \Sigma_0$, $\nabla_\tau(M) \subset t_\tau \cdot M$. We then have $\D \subset \mathrm{Sol}(M)$ so that $\D$ is $F$-analytic by the above. If $\D$ is an $F$-analytic $(\phi_q,\Gamma_K)$-module over $E\otimes_F\B_{\mathrm{rig},K}^{\dagger}$, then we have $\nabla_\tau(x) = 0$ for all $x \in \D$ by remark \ref{fla=nablatau nul} and so $\nabla_\tau(M) \subset t_\tau \cdot M$ for $M = (E \otimes_F \Bt_{\mathrm{rig},K}^\dagger)^{\pa}\otimes_{E\otimes_F\B_{\mathrm{rig},K}^{\dagger}} \D$ by lemma \ref{partial stable}.
\end{proof}

\section{$(B,E)$-pairs}
In this section, we quickly recall the definitions of $(B,E)$-pairs and define the notion of $F$-analytic $(B,E)$-pairs.

Let $\Bdrplus, \Bdr, \B_{\mathrm{cris}}^+$ and $\B_{\mathrm{cris}}$ be the usual Fontaine's rings of $p$-adic periods, defined for example in \cite{fontaine1994corps}. These rings come equipped with an action of $\G_{\Qp}$, and the rings $\B_{\mathrm{cris}}^+$ and $\B_{\mathrm{cris}}$ are endowed with an injective Frobenius $\phi$. We let $\B_e = (\B_{\mathrm{cris}})^{\phi=1}$. Berger defined in \cite{BerBpaires} the notion of $B$-pairs, that is pairs $W = (W_e,W_{dR}^+)$, where $W_e$ is a free $\B_e$-module of finite rank, equipped with a semilinear continuous action of $\G_K$ and where $W_{dR}^+$ is a $\G_K$-stable $\Bdrplus$-lattice inside $W_{dR} = \Bdr \otimes_{\B_e}W_e$. To a $p$-adic representation $V$, one can attach the $B$-pair $W(V) = (\B_e \otimes_{\Qp}V, \Bdrplus \otimes_{\Qp}V)$, and the functor $V \mapsto W(V)$ is fully faithful since $\B_e \cap \Bdrplus = \Qp$. Recall that $t$ is the usual $t$ in $p$-adic Hodge theory (note that $t$ corresponds to the element $t_p$ for $F=\Qp$) and that $\Bdrplus/t\Bdrplus = \Cp$.

Berger showed \cite[Thm. 2.2.7]{BerBpaires} how to attach to any $B$-pair a cyclotomic $(\phi,\Gamma)$-module $D(W)$ on the (cyclotomic) Robba ring, and that this functor induces an equivalence of categories.  

Let $E$ be a field of coefficients as previously. Let $\B_{e,E} = E \otimes_{\Qp}\B_e$, $\B_{\mathrm{dR},E}^+ = E \otimes_{\Qp} \Bdrplus$ and $\B_{\mathrm{dR},E} = E \otimes_{\Qp} \Bdr$, where $\G_{\Qp}$ acts $E$-linearly on $E$. A $(B,E)$-pair is a pair $W = (W_e,W_{dR}^+)$, where $W_e$ is a free $\B_{e,E}$-module of finite rank, equipped with a semilinear continuous action of $\G_K$ and where $W_{dR}^+$ is a $\G_K$-stable $\B_{\mathrm{dR},E}^+$-lattice inside $W_{dR} = \B_{\mathrm{dR},E} \otimes_{\B_{e,E}}W_e$. To an $E$ representation $V$, one can attach the $(B,E)$-pair $W(V) = (\B_e \otimes_{\Qp}V, \Bdrplus \otimes_{\Qp}V)$, and this functor is once again fully faithful. Theorem 2.2.7 of \cite{BerBpaires} has been extended by Nakamura \cite[Thm. 1.36]{Nakapieds} for $(B,E)$-pairs and cyclotomic $E$-$(\phi,\Gamma)$-modules, that is $(\phi,\Gamma)$-modules over the cyclotomic Robba ring tensored by $E$ over $\Qp$.

Let $F, E$ be as in \S 1. Note that we have an isomorphism $E \otimes_{\Qp}F \simeq \prod\limits_{\tau \in \Sigma}E$, given by $a \otimes b \mapsto (a \tau(b))_{\tau \in \Sigma}$. Since $F \subset \Bdrplus$, we have natural isomorphisms 
$$E \otimes_{\Qp}\Bdrplus \simeq (E \otimes_{\Qp} F) \otimes_F \Bdrplus \simeq (\prod_{\tau \in \Sigma}E) \otimes_F \Bdrplus \simeq \prod\limits_{\tau \in \Sigma}\B_{\mathrm{dR},\tau}^+$$
where $\B_{\mathrm{dR},\tau}^+ = E \otimes_F^\tau \Bdrplus$, and
$$E \otimes_{\Qp}\Bdr \simeq \prod\limits_{\tau \in \Sigma}\B_{\mathrm{dR},\tau}$$
where $\B_{\mathrm{dR},\tau} = E \otimes_F^\tau \Bdr$.

Using these decompositions, we get decompositions $W_{dR}^+ \simeq \prod\limits_{\tau \in \Sigma}W_{dR,\tau}^+$ and $W_{dR} \simeq \prod\limits_{\tau \in \Sigma}W_{dR,\tau}$.

We say that a $(B,E)$-pair is $F$-analytic if for all $\tau \in \Sigma_0$, $W_{dR,\tau}^+/tW_{dR,\tau}^+$ is the trivial $\Cp$-semilinear representation of $\G_K$. The following lemma shows that this definition is compatible with the one of $F$-analytic representation:

\begin{lemm}
Let $V$ be an $E$-representation of $\G_K$. Then $V$ is $F$-analytic if and only if the $(B,E)$-pair $W(V) = (W_e, W_{\mathrm{dR}}^+) =  (\B_e \otimes_{\Qp}V, \Bdrplus \otimes_{\Qp}V)$ is $F$-analytic.
\end{lemm}
\begin{proof}
We have $\Bdrplus/t\Bdrplus = \Cp$, so that $W_{dR}^+/tW_{dR}^+ = \Cp \otimes_{\Qp}V \simeq \prod\limits_{\tau \in \Sigma}(\Cp \otimes_F^\tau V)$, and $W_{dR,\tau}^+/tW_{dR,\tau}^+ = \Cp \otimes_F^\tau V$, and so the equivalence is clear.
\end{proof}

\begin{lemm}
\label{lemm BeE Brig1/t}
We have $\B_{e,E} = E \otimes_F (\Bt_{\mathrm{rig}}^{\dagger}[1/t])^{\phi_q=1}$. 	
\end{lemm}
\begin{proof}
First, recall that $\B_e = (\Bt_{\mathrm{rig},\eta}^\dagger[1/t])^{\phi=1}$ (this is \cite[Lemm. 1.1.7]{BerBpaires}). Since $\phi_q$ is $F$-linear, we have $(\Bt_{\mathrm{rig}}^\dagger[1/t])^{\phi_q=1}=(F \otimes_{F_0}\Bt_{\mathrm{rig},\eta}^\dagger[1/t])^{\phi_q=1} = F \otimes_{F_0} (\Bt_{\mathrm{rig},\eta}^\dagger[1/t])^{\phi^h=1}$. Now since $\Gal(F_0/\Qp)$ acts $F_0$-semi-linearly on $(\Bt_{\mathrm{rig},\eta}^\dagger[1/t])^{\phi^h=1}$ by $\phi$, Speiser's lemma implies that $(\Bt_{\mathrm{rig},\eta}^\dagger[1/t])^{\phi^h=1}= F_0 \otimes_{\Qp}\B_e$. Thus, we get that 
$$\B_{e,E} = E \otimes_{\Qp} \B_e = E \otimes_{F} F \otimes_{F_0} (F_0 \otimes \B_e)$$
and what we just did implies that
$$\B_{e,E} = E \otimes_F (\Bt_{\mathrm{rig}}^\dagger[1/t])^{\phi_q=1},$$
which is what we wanted.

Note that this is basically the same proof (tensored by $E$ over $F$) as \cite[Prop. 2.2]{BMTensTriang} but using $\Bt_{\mathrm{rig}}^\dagger[1/t]$ instead of $F \otimes_{F_0} \B_{\mathrm{cris}}$.
\end{proof}

\begin{lemm}
\label{prod t_tau}~
\begin{enumerate}
\item The $t$-adic valuation of the $\tau'$-component of the image of $t_\tau$ by the map $\Bt_{\mathrm{rig}}^+ \rightarrow F\otimes_{\Qp}\Bdr = \prod_{\tau' \in \Sigma}\Bdr$ given by $x \mapsto \{(\tau' \otimes\phi^{n(\tau')})(x)\}_{\tau' \in \Sigma}$ is $1$ if $\tau' = \tau^{-1}$ and $0$ otherwise. 
\item There exists $u \in (F \otimes \hat{\Q_p}^{\mathrm{unr}})^{\times}$ such that $\prod_{\tau \in \Sigma}t_\tau = u \cdot t$ in $\Bt_{\mathrm{rig}}^+$.
\end{enumerate}
\end{lemm}
\begin{proof}
These are items 2 and 3 of \cite[Prop. 2.4]{BMTensTriang}, using $\Bt_{\mathrm{rig}}^+$ instead of $F\otimes_{F_0}\B_{\mathrm{cris}}^+$.
\end{proof}

We now explain how to attach $F$-analytic $(B,E)$-pairs to $F$-analytic $(\phi_q,\Gamma_K)$-modules over $(E \otimes_F \B_{\mathrm{rig},K}^\dagger)$ and vice-versa.

Lemma \ref{lemm BeE Brig1/t} allows us to see $E \otimes_F \Bt_{\mathrm{rig}}^{\dagger}[1/t]$ as a $\B_{e,E}$-module.

Let $\Omega = \{ (\tau,n) \in \Gal(E/\Qp) \times \Z$ such that $n(\tau{\mid_F}) \equiv n \bmod{h}\}$. For $n \geq 0$, let $r_n = p^{n-1}(p-1)$, and for $r > 0$, let $n(r)$ be the least integer $n$ such that  $r_n \geq r$. For $r \geq 0$, we let $\Omega_r = \{ (\tau,n) \in \Omega$ such that $n \geq n(r)\}$. For $g =(\tau,n) \in \Omega$, we let $\tau(g) = \tau$ and $n(g) = n$. If $\min(I) \geq r$ and if $g \in \Omega_r$, we have a map $\iota_g : E \otimes_F \Bt^I \to E \otimes_F^{\tau(g)_{|F}} \Bdrplus = \B_{\mathrm{dR},\tau(g)_{|F}}$, defined in \cite[\S 5]{Ber14MultiLa} and given by $x \mapsto (g^{-1} \otimes (g{|_F^{-1}} \otimes \phi^{-n(g)})) (x)$. 

\begin{lemm}
\label{lemm intermediaire Bpairs to phigamma}
Let $W$ be a $(B,E)$-pair of rank $d$, and let 
$$\tilde{D}^r(W)= \left\{y \in (E\otimes_F\Bt_{\mathrm{rig}}^{\dagger,r}[1/t])\otimes_{\B_{e,E}}W_e \textrm{ such that } \iota_g(y) \in W_{dR,\tau(g)_{|F}}^+ \textrm{ for all } g \in \Omega_r \right\}.$$ Then:
\begin{enumerate}
\item $\tilde{D}^r(W)$ is a free $E\otimes_F\Bt_{\mathrm{rig}}^{\dagger,r}$-module of rank $d$;
\item $\tilde{D}^r(W)[1/t] = (E\otimes_F\Bt_{\mathrm{rig}}^{\dagger,r}[1/t])\otimes_{\B_{e,E}}W_e$.
\end{enumerate}
\end{lemm}
\begin{proof}
This is \cite[Lemm. 2.2.1]{BerBpaires} tensored by $E$.
\end{proof}

If $W$ is a $(B,E)$-pair, we let $\tilde{D}(W) = (E\otimes_F\Bt_{\mathrm{rig}}^{\dagger}) \otimes_{E\otimes_F\Bt_{\mathrm{rig}}^{\dagger,r}}\tilde{D}^r(W)$, and if $I$ is a subinterval of $[r;+\infty[$, we let $\tilde{D}^I(W) = (E \otimes_F \Bt^I)\otimes_{E\otimes_F\Bt_{\mathrm{rig}}^{\dagger,r}}\tilde{D}^r(W)$. By the same argument as in \cite[Lemm. 2.2.2]{BerBpaires}, this does not depend on the choice of $r \in I$. 

\begin{prop}
\label{phigamma eta inside Dtilde}
If $W$ is a $(B,E)$-pair of rank $d$, then there exists a unique $(\phi_q,\Gamma_K')$-module $\D_\eta(W)$ over $E \otimes_{F_0} \B_{\mathrm{rig},K,\eta}^\dagger$ such that $(E \otimes_{F_0} \Bt_{\mathrm{rig}}^\dagger) \otimes_{E \otimes_{F_0} \B_{\mathrm{rig},K,\eta}^\dagger}\D_\eta(W) = \tilde{D}(W)$.
\end{prop}
\begin{proof}
This is \cite[Prop. 2.2.5]{BerBpaires} up to a tensor product, and using the twisted cyclotomic case instead of the classical one, but again by using \cite[\S 8]{Ber14MultiLa}, it does not change the arguments of the proof.
\end{proof}

For $r \geq 0$ such that $\D_\eta(W)$ and all its structures are defined over $E \otimes_{F_0} \B_{\mathrm{rig},K,\eta}^{\dagger,r}$, we let $\D_\eta^r(W)$ be the associated $(E \otimes_{F_0} \B_{\mathrm{rig},K,\eta}^{\dagger,r})$-module so that $\D_\eta(W) = (E \otimes_F \B_{\mathrm{rig},K,\eta}^\dagger) \otimes_{E \otimes_{F_0} \B_{\mathrm{rig},K,\eta}^{\dagger,r}} \D_\eta^r(W)$. For $I = [r;s]$, we let $\D_\eta^I= (E \otimes_{F_0} \B_{K,\eta}^I) \otimes_{E \otimes_{F_0} \B_{\mathrm{rig},K,\eta}^{\dagger,r}} \D_\eta^r(W)$. Let $\tilde{D}^I_K(W) = (\tilde{D}^I(W))^{H_K}$ and $\tilde{D}_K(W) = \tilde{D}(W)^{H_K}$, so that $\tilde{D}^I_K(W) = (E\otimes_F \Bt_{K}^I)\otimes_{E \otimes_{F_0} \B_{K,\eta}^I}\D_\eta^I(W)$ and $\tilde{D}_K(W) = (E\otimes_F \Bt_{\mathrm{rig},K}^\dagger)\otimes_{E \otimes_{F_0} \B_{\mathrm{rig},K,\eta}^\dagger}\D_\eta(W)$ (since $\D_\eta(W)$ is invariant under $H_K$).

\begin{prop}
We have
\begin{enumerate}
\item $\tilde{D}_K^I(W)^{\la} = (E\otimes_F \Bt_{K}^I)^{\la}\otimes_{E \otimes_F \B_{K,\eta}^I}\D_\eta^I(W)$;
\item $\tilde{D}_K(W)^{\pa} = (E\otimes_F \Bt_{\mathrm{rig},K}^\dagger)^{\pa}\otimes_{E \otimes_F \B_{\mathrm{rig},K,\eta}^\dagger}\D_\eta(W)$.
\end{enumerate}
\end{prop}
\begin{proof}
The same proof as \cite[\S 2.1]{KR09} shows that the elements of $\D_\eta^I(W)$ are locally analytic vectors, and the result now follows from proposition \ref{lainla and painpa}.
\end{proof}

\begin{theo}
\label{FanBpair gives FanphiGamma}
If $W$ is an $F$-analytic $(B,E)$-pair of rank $d$, then there exists a unique $F$-analytic $(\phi_q,\Gamma_K)$-module $D(W)$ over $E \otimes_F \B_{\mathrm{rig},K}^\dagger$ such that 
$$(E \otimes_F \Bt_{\mathrm{rig}}^\dagger) \otimes_{E \otimes_F \B_{\mathrm{rig},K}^\dagger}D(W) = \tilde{D}(W).$$
\end{theo}
\begin{proof}
Let $W$ be an $F$-analytic $(B,E)$-pair of rank $d$, and let $\tilde{D}_K(W)$ be as above. Let $r \geq 0$ and let $y \in (\tilde{D}_K(W)^r)^{\pa}$. Let $\tau \in \Sigma \setminus \{\id\}$ and let 
$$\Omega_{\tau,r} = \left\{g \in \Omega \textrm{ such that } n(g) \geq n(r) \textrm{ and } \tau(g)=\tau \right\}.$$
Let $g \in \Omega_{\tau,r}$. We have $\iota_g(y) \in W_{dR,\tau}^+$. Write $x_g$ for the image of $\iota_g(y)$ in $W_{dR,\tau}^+/tW_{dR,\tau}^+$. Since the filtration on $W_{dR,\tau}$ is Galois stable, we get that $x_g$ is invariant under $H_K$ (since $\iota_g(y)$ is), and is a locally analytic vector of $(W_{dR,\tau}^+/tW_{dR,\tau}^+)^{H_K}$ using the fact that $y \in (\tilde{D}_K(W)^r)^{\pa}$. Note that $\nabla_\id = 0$ on $((W_{dR,\tau}^+/tW_{dR,\tau}^+)^{H_K})^{\la}$ since $W$ is $F$-analytic and by \cite[Prop. 2.10]{Ber14MultiLa}. This shows that $\nabla_\id(x_g) = 0$ and so $\nabla_\id(\iota_g(y))=0 \mod t_\pi$ (recall that $t$ and $t_\pi$ both generate the kernel of $\theta$ in $\Bdrplus$ by lemma \ref{prod t_tau}). Using the fact that $\iota_g \circ \nabla_\tau = \nabla_\id \circ \iota_g$, this implies that $t_\pi | \iota_g\circ \nabla_\tau(y)$ in $W_{dR,\tau}^+$. By lemma \ref{prod t_tau}, this proves that $\iota_g((Q_n^\tau)^{-1}\cdot \nabla_\tau(y)) \in W_{dR,\tau}^+$ for $n=n(g)$. By definition of $\tilde{D}^r(W)$, this proves that $\nabla_\tau(y) \in Q_n^\tau \cdot \tilde{D}^r(W)$ for all $n \geq n(r)$, and so $\nabla_\tau$ is divisible by $\prod\limits_{n=n(r)}^{+\infty}Q_n^\tau$ in $\tilde{D}^r(W)$ (the argument for the divisibility by an infinite product is the same as the one given in the proof of \cite[Lemm. 10.2]{Ber14MultiLa}), hence by $t_\tau$.

In particular, for all $\tau \in \Sigma \setminus \{\id\}$, we have $\nabla_\tau(\tilde{D}^r(W)^{\pa}) \subset t_\tau \cdot \tilde{D}^r(W)^{\pa}$. By proposition \ref{prop eq phigamma fan}, there exists a unique $(\phi_q,\Gamma_K)$-module $\D_{\mathrm{rig}}^\dagger$ over $E \otimes_F \B_{\mathrm{rig},K}^\dagger$ such that $(E \otimes_F \Bt_{\mathrm{rig}}^\dagger) \otimes_{E \otimes_F \B_{\mathrm{rig},K}^\dagger}\D_{\mathrm{rig}}^\dagger = \tilde{D}(W)$, which is what we wanted.
\end{proof}

We now explain how to attach an $F$-analytic $(B,E)$-pair to an $F$-analytic $(\phi_q,\Gamma_K)$-module.

\begin{prop}
If $\D$ is a $\phi_q$-module over $\B_{\mathrm{rig},K}^\dagger$, then there exists $r(\D) \geq r(K)$ such that, for all $r \geq r(\D)$, there exists a unique sub $\B_{\mathrm{rig},K}^{\dagger,r}$-module $\D_r$ of $\D$ such that:
\begin{enumerate}
\item $\D = \B_{\mathrm{rig},K}^\dagger \otimes_{\B_{\mathrm{rig},K}^{\dagger,r}}\D_r$;
\item the $\B_{\mathrm{rig},K}^{\dagger,qr}$-module $\B_{\mathrm{rig},K}^{\dagger,qr}\otimes_{\B_{\mathrm{rig},K}^{\dagger},r}\D_r$ has a basis contained inside $\phi_q(\D)$.
Moreover, if $\D$ is a $(\phi_q,\Gamma_K)$-module, one has $g(\D_r) = \D_r$ for all $g \in \Gamma_K$.
\end{enumerate}
\end{prop}
\begin{proof}
This is exactly the same proof as \cite[Thm. I.3.3]{BerBpaires} but using Lubin-Tate $(\phi_q,\Gamma_K)$-modules instead of cyclotomic ones, and tensoring by $E$ over $F$.
\end{proof}

\begin{prop}
\label{prop phigamma to Bpair}
If $\D$ is a $(\phi_q,\Gamma_K)$-module over $E \otimes_F \B_{\mathrm{rig},K}^{\dagger}$, free of rank $d$, then 
\begin{enumerate}
\item $W_e(\D) = (E \otimes_F \Bt_{\mathrm{rig},K}^\dagger[1/t]\otimes_{\B_{\mathrm{rig},K}^{\dagger}}\D)^{\phi_q=1}$ is a free $\B_{e,E}$-module of rank $d$ which is $\G_K$-stable;
\item $W_{dR}^+= \prod_{\tau \in \Sigma}\left((E \otimes_F \Bdrplus) \otimes_{E \otimes_F\B_{\mathrm{rig},K}^{\dagger,r_{n(g)}}}^{\iota_g}\D^{r_{n(g)}}\right)_{g \in \Omega_{r,\tau}}$ does not depend on $n(g) \gg 0$ and is a free $E \otimes_{\Qp}\Bdrplus = (\Bdrplus_{\tau})_{\tau \in \Sigma}$-module of rank $d$ and $\G_K$-stable. 
\item $W(\D) = (W_e(\D),W_{dR}^+(\D))$ is a $(B,E)$-pair.
Moreover, if $\D$ is $F$-analytic, then so is $W(D)$.
\end{enumerate}
\end{prop}
\begin{proof}
The proof of items 1 and 2 is the same as \cite[Prop. 2.2.6]{BerBpaires}. Assume now that $\D$ is $F$-analytic, and let us prove that $W(\D)$ is $F$-analytic. Let $\tau \in \Sigma \setminus \{\id\}$. 

By item 2, we have $W_{dR,\tau}^+=(E \otimes_F \Bdrplus) \otimes_{E \otimes_F\B_{\mathrm{rig},K}^{\dagger,r_{n(g)}}}^{\iota_g}\D^{r_{n(g)}}$ for some $g \in \Omega_{r,\tau}$. We can find a basis $e_1,\ldots,e_d$ of $\D^{r_{n(g)}}$ over $E \otimes_F \B_{\mathrm{rig},K}^{\dagger,r_{n(g)}}$ such that the image of the basis $\iota_g(e_1),\ldots,\iota_g(e_d)$ of $W_{dR,\tau}^+$ over $E \otimes_F \Bdrplus$ modulo $t_\pi$ is a basis of the $E\otimes_F\Cp$-representation $W_{dR,\tau}^+/tW_{dR,\tau}^+$.

Since the $e_i$ are pro-analytic vectors of $\D^{r_{n(g)}}$ for the action of $\Gamma_K$, the same argument as in the proof of theorem \ref{FanBpair gives FanphiGamma} shows that their image in $W_{dR,\tau}^+/tW_{dR,\tau}^+$ are invariant under $H_K$ and locally analytic vectors of $(W_{dR,\tau}^+/tW_{dR,\tau}^+)^{H_K}$. Since 
$$\nabla_\tau\left((E \otimes_F\Bt_{\mathrm{rig},K}^{\dagger,r_{n(g)}})^{\pa}\otimes_{E \otimes_F\B_{\mathrm{rig},K}^{\dagger,r_{n(g)}}}\D^{r_{n(g)}}\right) \subset t_\tau \cdot \left((E \otimes_F\Bt_{\mathrm{rig},K}^{\dagger,r_{n(g)}})^{\pa}\otimes_{E \otimes_F\B_{\mathrm{rig},K}^{\dagger,r_{n(g)}}}\D^{r_{n(g)}}\right)$$
by lemma \ref{fla=nablatau nul} and since 
$$W_{dR,\tau}^+ = (E \otimes_F \Bdrplus) \otimes_{E \otimes_F\B_{\mathrm{rig},K}^{\dagger,r_{n(g)}}}^{\iota_g}((E \otimes_F\Bt_{\mathrm{rig},K}^{\dagger,r_{n(g)}})^{\pa}\otimes_{E \otimes_F\B_{\mathrm{rig},K}^{\dagger,r_{n(g)}}}\D^{r_{n(g)}})$$
we get that $\nabla_\id(e_i) = 0$ mod $t_\pi$ for all $i$ since $\iota_g \circ \nabla_\tau = \nabla_\id \circ \iota_g$ and since $\iota_g(t_\tau) = t_\pi$.

This implies that $\nabla_\id = 0$ on $(W_{dR,\tau}^+/tW_{dR,\tau}^+)^{H_K, \la}$ so that $(W_{dR,\tau}^+/tW_{dR,\tau}^+)$ is $\Cp$-admissible as a $E \otimes_F \Cp$ representation of $\G_K$, using the discussion following \cite[Thm. 4.11]{Ber14SenLa}.
\end{proof}

\begin{theo}
The two functors $W \mapsto D(W)$ and $\D \mapsto W(\D)$ are inverse one to another and induce an equivalence of categories between the category of $F$-analytic $(B,E)$-pairs and the category of $F$-analytic $(\phi_q,\Gamma_K)$-modules.
\end{theo}
\begin{proof}
Let $W=(W_e,W_{dR}^+)$ be an $F$-analytic $(B,E)$-pair and let $\D = D(W)$. By definition of $W(\D)$, we have 
$$(E\otimes_F \Bt_{\mathrm{rig}}^\dagger[1/t])\otimes_{\B_{e,E}}W_e(\D) = (E\otimes_F \Bt_{\mathrm{rig}}^\dagger[1/t]) \otimes_{E \otimes_F\B_{\mathrm{rig},K}^{\dagger}} \D$$
and by definition of $D(W)$, we have
$$(E\otimes_F \Bt_{\mathrm{rig}}^\dagger[1/t])\otimes_{\B_{e,E}}W_e = (E\otimes_F \Bt_{\mathrm{rig}}^\dagger[1/t]) \otimes_{E \otimes_F\B_{\mathrm{rig},K}^{\dagger}} \D$$
so that, taking the invariants by $\phi_q$, we get that $W_e \simeq W(\D)$ as $\B_{e,E}$-representations.

Let $\tau \in \Sigma$. By definition of $W_{dR,\tau}^+(\D)$, we have $W_{dR,\tau}^+(\D) = (E \otimes_F\Bdrplus)\otimes^{\iota_g}\D^{r_{n(g)}}$ for some $g \in \Omega_{r,\tau}$ with $r$ big enough, and hence
$$W_{dR,\tau}^+(\D) = (E \otimes_F\Bdrplus)\otimes^{\iota_g}\tilde{D}^{r_{n(g)}}$$ 
where $\tilde{D}^r = \tilde{D}^r(W) = (E\otimes_F \Bt_{\mathrm{rig}}^{\dagger,r})\otimes_{E \otimes_F \B_{\mathrm{rig},K}^{\dagger,r}}\D^r$ by proposition \ref{phigamma eta inside Dtilde}. Recall that 
$$\tilde{D}^r(W)=\left\{y \in (E\otimes_F\Bt_{\mathrm{rig}}^{\dagger,r}[1/t])\otimes_{\B_{e,E}}W_e \textrm{ such that } \iota_g(y) \in W_{dR,\tau(g)_{|F}}^+ \textrm{ for all } g \in \Omega_r \right\},$$
so that, after tensoring by $E\otimes_F\Bdrplus$ over $\iota_g$, we get $W_{dR,\tau}^+(\D(W))=W_{dR,\tau}^+$.

Let $\D$ be an $F$-analytic $(\phi_q,\Gamma_K)$-module and let $W=W(\D)$ and $\tilde{D} = (E \otimes_F\B_{\mathrm{rig}}^\dagger) \otimes_{E\otimes_F\B_{\mathrm{rig},K}^\dagger}\D$. The same reasoning as above shows that
$$(E\otimes_F\Bt_{\mathrm{rig}}^\dagger[1/t])\otimes_{E \otimes_F\B_{\mathrm{rig},K}^{\dagger}}\D = (E\otimes_F\Bt_{\mathrm{rig}}^\dagger[1/t])\otimes_{E \otimes_F\B_{\mathrm{rig},K}^{\dagger}}\D(W(\D))$$
and that
$$(E\otimes_F\Bt_{\mathrm{rig}}^\dagger[1/t])\otimes_{E \otimes_F\B_{\mathrm{rig},K}^{\dagger}}\tilde{D} = (E\otimes_F\Bt_{\mathrm{rig}}^\dagger[1/t])\otimes_{E \otimes_F\B_{\mathrm{rig},K}^{\dagger}}\tilde{D}(W(\D)).$$

If $M$ is a $(\phi_q,\Gamma_K)$-module over $E \otimes_F\B_{\mathrm{rig}}^\dagger$, note that we can recover $M$ inside $M[1/t]$ by
$$M = \left\{x \in M[1/t] \textrm{ such that } \iota_g(x) \in (E\otimes_F\Bdrplus) \otimes_{E\otimes_F \Bt_{\mathrm{rig}}^{\dagger}}^{\iota_g}M \textrm{ for all } g \textrm{ with } n(g) \gg 0 \right\}.$$ 
In particular, since 
$$(E\otimes_F\Bt_{\mathrm{rig}}^\dagger[1/t])\otimes_{E \otimes_F\B_{\mathrm{rig},K}^{\dagger}}\tilde{D} = (E\otimes_F\Bt_{\mathrm{rig}}^\dagger[1/t])\otimes_{E \otimes_F\B_{\mathrm{rig},K}^{\dagger}}\tilde{D}(W(\D)),$$
this shows that
$$\tilde{D} = \tilde{D}(W(\D)).$$
Since $\D$ is $F$-analytic, we have $\nabla_\tau((\tilde{D}_K)^{\pa}) \subset t_\tau \cdot (\tilde{D}_K)^{\pa})$ for all $\tau \in \Sigma \setminus \{\id\}$ by proposition \ref{prop eq phigamma fan}, hence there exists, still by proposition \ref{prop eq phigamma fan}, a unique $F$-analytic $(\phi_q,\Gamma_K)$-module $\D_{\mathrm{rig}}^\dagger$ over $E\otimes_F\B_{\mathrm{rig},K}^\dagger$ such that 
$$\mathrm{Sol}(\tilde{D}_K^{\pa}) = (E \otimes_F (\Bt_{\mathrm{rig},K}^\dagger)^{\fpa}) \otimes_{E \otimes_F \Bt_{\mathrm{rig},K}^\dagger} \D_{\mathrm{rig}}^\dagger$$
 and such that 
$$\tilde{D} = (E \otimes_F \Bt_{\mathrm{rig},K}^\dagger)\otimes_{E\otimes_F\B_{\mathrm{rig},K}^{\dagger}} \D_{\mathrm{rig}}^\dagger$$

In particular, we have $\D=\D(W(\D))=\D_{\mathrm{rig}}^\dagger$, which concludes the proof.
\end{proof}

\section{A simpler equivalence in the $F$-analytic case}
In this section, we quickly recall the constructions of Ding's $\B_\sigma$-pairs and explain the corresponding equivalence of category that follows between $F$-analytic $(\phi_q,\Gamma_K)$-modules and $B_{\id}$-pairs.

Let $\B_{e,F}^{\LT} = (\Btrigplus[1/t_\pi])^{\phi_q=1}$. As usual, it is easy to check that $\B_{e,F}^{\LT} = (\Bt_{\rig}^\dagger[1/t_\pi])^{\phi_q=1}$. Following \cite{Dingpartially}, we make the following definition:

\begin{defi}
\begin{enumerate}
\item Let $\sigma \in \Sigma_E$ be any embedding. A $B_\sigma$-pair is the data of a couple $W_\sigma = (W_{\sigma,E}^{\LT},W_{\dR,\sigma}^+)$ where $W_{\sigma,E}^{\LT}$ is a finite free $E \otimes_F^\sigma \B_{e,F}^{\LT}$-module equipped with a semi-linear $\G_K$ action and $W_{\dR,\sigma}^+$ is a $\G_K$-invariant $\B_{\dR,\sigma}^+$-lattice in $W_{\dR,\sigma} := W_{\sigma,E}^{\LT} \otimes_{E \otimes_F^\sigma \B_{e,F}^{\LT}} \B_{\dR,\sigma}$.
\item For two $B_{\sigma}$-pairs $W_\sigma, W_\sigma'$, a morphism $f: W_\sigma \ra W_\sigma'$ is a $\G_K$-invariant $E \otimes_F^\sigma \B_{e,F}^{\LT}$-linear map $f_{\sigma,E}^{\LT} : W_{\sigma,E}^{\LT} \ra (W_{\sigma,E}')^{\LT}$ such that the induced $\B_{\dR,\sigma}$-linear map $f_{\dR,\sigma}:=f_{\sigma,E}^{\LT} \otimes \id : W_{\dR,\sigma} \ra W_{\dR,\sigma}'$ sends $W_{\dR,\sigma}^+$ to $(W')_{\dR,\sigma}^+$.
\end{enumerate}
\end{defi}

Let $W = (W_e,W_{\dR}^+)$ be a $(B,E)$-pair. Let $W_{\sigma,E}^{\LT} = $
$$\left\{ w \in W_e : \tau(w) \in W_{\dR,\sigma \circ \tau^{-1}}^+ \textrm{ for all } \tau \in \Gal(E/\Qp), \tau_{|F} \neq \id \right\}.$$
By \cite[Lemm. 1.3]{Dingpartially}, this is a $E \otimes_F^\sigma \B_{e,F}^{\LT}$-module.

\begin{prop}
\label{prop construction F_sigma Ding}
For $\sigma \in \Sigma_E$, the functor
$$F_\sigma : \{(B,E)-\textrm{pairs}\} \ra \{B_\sigma-\textrm{pairs}\}$$
given by 
$$W=(W_e,W_{\dR}^+) \mapsto W_\sigma=(W_{\sigma,E}^{\LT},W_{\dR,\sigma}^+)$$
induces an equivalence of categories.
\end{prop}
\begin{proof}
This is \cite[Prop. 3.7]{Dingpartially}.
\end{proof}

For $\sigma \in \Sigma_E$, let $G_\sigma$ denote the inverse functor of $F_\sigma$ defined by Ding in \cite[Lemm. 3.8]{Dingpartially}. We say that a $B_{\id}$-pair $W$ is $F$-analytic if for all $\sigma \in \Sigma_E$ such that $\sigma_{|F}\neq \id_F$, then $W_{\dR,\sigma}^+/tW_{\dR,\sigma}^+$ is the trivial $\Cp$-representation of $\G_K$, where $W_{\dR,\sigma}^+$ is the second component of the $B_\sigma$-pair $F_\sigma \circ G_{\id}(W)$. By \cite[Lemm. 3.9]{Dingpartially}, this is the same as asking that the corresponding $(B,E)$-pair $G_{\id}(W)$ is $F$-analytic.

\begin{prop}
If $\D$ is a $(\phi_q,\Gamma_K)$-module over $E \otimes_F \B_{\mathrm{rig},K}^{\dagger}$, free of rank $d$, then 
\begin{enumerate}
\item $W_{\id,E}^{\LT}(\D) = (E \otimes_F \Bt_{\mathrm{rig},K}^\dagger[1/t_\pi]\otimes_{\B_{\mathrm{rig},K}^{\dagger}}\D)^{\phi_q=1}$ is a free $E \otimes_F^\sigma \B_{e,F}^{\LT}$-module of rank $d$ which is $\G_K$-stable;
\item $W_{\dR,\id}^+= \left((E \otimes_F \Bdrplus) \otimes_{E \otimes_F\B_{\mathrm{rig},K}^{\dagger,r_{n(g)}}}^{\iota_g}\D^{r_{n(g)}}\right)_{g \in \Omega_{\id,r}}$ does not depend on $n(g) \gg 0$ and is a free $\B_{\dR,\id}^+$-module of rank $d$ which is $\G_K$-stable. 
\item $W(\D)^{\LT} = (W_{\id,E}^{\LT}(\D),W_{\dR,\id}^+(\D))$ is a $B_{\id}$-pair.
Moreover, if $\D$ is $F$-analytic, then so is $W(\D)$.
\end{enumerate}
\end{prop}
\begin{proof}
The proof of items 1, 2 and 3 is the same as in \ref{prop phigamma to Bpair}. The part on $F$-analyticity now follows from the remark above and the fact (which follows easily from proposition \ref{prop construction F_sigma Ding}) that the $B_{\id}$-pair $W(\D)$ we just constructed is exactly $F_{\id}(W')$ where $W'$ is the $(B,E)$-pair attached to $\D$ constructed in proposition \ref{prop phigamma to Bpair}. 
\end{proof}

In particular, this construction is the exact analogue of Berger's construction \cite[Prop. 2.2.6]{BerBpaires} in the cyclotomic case.

We now explain how to recover the $(\phi_q,\Gamma_K)$-module $\D_{\rig}^\dagger$ attached to an $F$-analytic $\B_{\id}$-pair $W$. Let $W_\id=(W_{\id,E}^{\LT},W_{\dR,\id}^+)$ be an $F$-analytic $B_\id$-pair.

\begin{lemm}
Let $\tilde{D}^r(W)^{\LT}= $
$$\left\{y \in (E\otimes_F\Bt_{\mathrm{rig}}^{\dagger,r}[1/t_\pi])\otimes_{E \otimes_F \B_{e,F}^{\LT}}W_{\id,E}^{\LT} \textrm{ such that } \iota_g(y) \in W_{\dR,\id}^+ \textrm{ for all } g \in \Omega_{r} \textrm{ with } \tau(g) = \id \right\}.$$
Then:
\begin{enumerate}
\item $\tilde{D}^r(W)^{\LT}$ is a free $E \otimes_F \Bt_{\rig}^{\dagger,r}$-module of rank $d$;
\item $\tilde{D}^r(W)^{\LT}[1/t_\pi] = E \otimes_F \Bt_{\rig}^{\dagger,r}[1/t_\pi]\otimes_{E \otimes_F \B_{e,F}^{\LT}}W_{\id,E}^{\LT}$.
\end{enumerate}
\end{lemm}
\begin{proof}
This is the same proof as in lemma \ref{lemm intermediaire Bpairs to phigamma} but here we do not need to keep track of all the embeddings. 
\end{proof}
 
We know that there are enough pro-analytic vectors inside $\tilde{D}(W)^{\LT}$, just because we already know by the constructions of \S 5 that it contains the $F$-analytic $(\phi_q,\Gamma_K)$-module $D(W')$ attached to $W'=G_\id(W)$ of theorem \ref{FanBpair gives FanphiGamma}. We can now recover it by taking the pro-analytic vectors of $\tilde{D}(W)^{\LT}$ and taking the module $\D_{\rig}^\dagger(\tilde{D}(W)^{\LT})$ given by proposition \ref{prop eq phigamma fan}. In particular, the following is a straightforward consequence of our previous constructions:

\begin{theo}
The functors $\D \mapsto W(\D)^{\LT}$ and $W_\id \mapsto \D_{\rig}^\dagger(\tilde{D}(W)^{\LT})$ are inverse of each other an give rise to an equivalence of categories between the category of $F$-analytic $(\phi_q,\Gamma_K)$-modules and the category of $F$-analytic $B_{\id}$-pairs.
\end{theo}

\bibliographystyle{amsalpha}
\bibliography{bibli}
\end{document}